\documentclass[11pt]{article}

\usepackage{amsmath, amssymb, amsthm, verbatim,enumerate,enumitem,bbm,color} 
\usepackage{indentfirst}

\title{On MAXCUT in strictly supercritical random graphs, and coloring  of random graphs and random tournaments}

\author{Lior Gishboliner \thanks{School of Mathematical Sciences, Raymond and Beverly Sackler Faculty of Exact Sciences, Tel Aviv University, Tel Aviv, 6997801, Israel. Email: liorgis1@post.tau.ac.il. Research supported in part by ISF Grant 224/11 and ERC-starting Grant 633509.} \and Michael Krivelevich \thanks{School of Mathematical Sciences,
		Raymond and Beverly Sackler Faculty of Exact Sciences, Tel Aviv
		University, Tel Aviv, 6997801, Israel. Email:
		krivelev@post.tau.ac.il. Research supported in part by USA-Israel BSF Grant 2014361 and by
		grant 912/12 from the Israel Science Foundation.} 
	\and Gal Kronenberg
	\thanks{School of Mathematical Sciences, Raymond and Beverly Sackler Faculty of Exact Sciences, Tel Aviv University, Tel Aviv, 6997801, Israel. Email: galkrone@mail.tau.ac.il.}}

\date{\today}
\allowdisplaybreaks

\theoremstyle{plain}
\newtheorem{theorem}{Theorem}[section]
\newtheorem{lemma}[theorem]{Lemma}
\newtheorem{claim}[theorem]{Claim}

\newtheorem{observation}[theorem]{Observation}
\newtheorem{corollary}[theorem]{Corollary}
\newtheorem{conjecture}[theorem]{Conjecture}

\newtheorem{remark}[theorem]{Remark}
\newtheorem{definition}[theorem]{Definition}

\newcommand{\Bin}{\ensuremath{\textrm{Bin}}}

\newcommand{\whp}{w.h.p.\ }
\newcommand{\Whp}{W.h.p.\ }

\newcommand{\C}{\mathcal C}

\newcommand{\K}{\mathcal K}
\newcommand{\dist}{\mathrm{Dist}_{\text{BP}}}
\newcommand{\distT}{\mathrm{Dist}_{\text{Tour-BP}}}

\newcommand{\arc}{\overrightarrow}

\RequirePackage[normalem]{ulem} 
\RequirePackage{color}\definecolor{RED}{rgb}{1,0,0}\definecolor{BLUE}{rgb}{0,0,1} 

\begin{document}
\maketitle

\begin{abstract}
		We use a theorem by Ding, Lubetzky and Peres describing the structure of the giant component of random graphs in the strictly supercritical regime, in order to determine the typical size of MAXCUT of $G\sim G\left(n,\frac {1+\varepsilon}n\right)$ in terms of $\varepsilon$.  We then apply this result to prove the following conjecture by Frieze and Pegden. For every $\varepsilon>0$ there exists $\ell_\varepsilon$ such that \whp $G\sim G(n,\frac {1+\varepsilon}n)$ is not homomorphic to the cycle on $2\ell_\varepsilon+1$ vertices. We also consider the coloring properties of biased random tournaments. A $p$-random tournament on $n$ vertices is obtained from the transitive tournament by reversing each edge independently with probability $p$. We show that for $p=\Theta(\frac 1n)$ the chromatic number of a $p$-random tournament behaves similarly to that of a random graph with the same edge probability. To treat the case $p=\frac {1+\varepsilon}n$ we use the aforementioned result on MAXCUT and show that in fact \whp one needs to reverse $\Theta(\varepsilon^3)n$ edges to make it 2-colorable.
\end{abstract}

\section{Introduction}

Given a graph $G$, a {\em bipartition} of $G$ is a partition of $V(G)$ into two sets, $V(G) = V_1 \uplus V_2$. The {\em cut} of the partition $(V_1,V_2)$ is the set of edges with one end-point in each $V_i$. The MAXCUT problem asks to find the size of a largest cut in $G$. We denote this number by $\text{MAXCUT}(G)$. The problem of finding $\text{MAXCUT}(G)$ has been extensively studied. It is known to be very important in both combinatorics and theoretical computer science, and has some connections to physics.  The MAXCUT problem is known to be NP-hard (see \cite{GJS,Trevisan2000}) and even not approximable to within a factor of $\frac{16}{17}$ unless $P = NP$ (see \cite{MAXCutIsNPH}). On the other hand, as shown by Goemans and Williamson \cite{GoeWill95}, there is a semidefinite programming algorithm that approximates MAXCUT to a factor of 0.87856. Moreover, for dense graphs there are polynomial time approximation
schemes for $\text{MAXCUT}(G)$, which approximate it up to an additive factor of $o(n^2)$, as shown by Arora, Karger and Karpinski in \cite{DenseMaxCut1}, and by Frieze and Kannan in \cite{DenseMaxCut2}.

One natural generalization of the MAXCUT problem is the MAX $k$-CUT problem that asks an analogous question about $k$-partitions of a graph. A $k$-partition of $G$ is a partition of $V(G)$ into $k$ sets, 
$V(G) = V_1 \uplus \dots \uplus V_k$.  The $k$-cut of the partition $(V_1,\dots,V_k)$ is the set of edges connecting vertices in different parts. The MAX $k$-CUT problem asks to find the size of a largest $k$-cut in $G$. Extending the methods of \cite{GoeWill95}, Frieze and Jerrum found in \cite{FriezeJerrum} an algorithm that approximates MAX $k$-CUT up to a known constant factor $\alpha_k < 1$. 

For any graph $G = (V,E)$, finding $\text{MAXCUT}(G)$ is clearly equivalent to finding the maximal number of edges in a bipartite subgraph of $G$. The {\em distance of $G$ from bipartiteness} is defined to be the minimal number of edges whose removal turns $G$ into a bipartite graph. We denote the distance by $\dist(G)$. Obviously,
$\dist(G) = |E| - \text{MAXCUT}(G)$. So asking for the size of a maximum cut in $G$ is equivalent to asking for the distance of $G$ from bipartiteness. 

In this paper we consider the maximum cut in a random graph. We work in the binomial random graph model $G(n,p)$. This is the probability space that consists of all graphs with $n$ labeled vertices, where each one of the $\binom{n}{2}$ possible edges is included independently with probability $p$. The study of the typical value of MAXCUT in the model $G(n,p)$ has a long history. For starters, if $p$ is not too small, say 
$p \gg n^{-1/2}\log n$, then one can check (using simple probabilistic tools) that the size of the MAXCUT of a typical $G\sim G(n,p)$ is 
$\frac {n^2p}4(1+o(1))$. The problem is more interesting for smaller $p$, specifically $p = \Theta\left( \frac{1}{n} \right)$. It was shown in \cite{Feige} that a random graph $G(n,p)$ with $p = \frac cn$ (for a large enough constant $c$), does not have a cut with significantly more than half the edges. In 2013, Bayati, Gamarnik and Tetali \cite{BGT} proved that for any $c > 0$, the random variable 
$n^{-1}\cdot \text{MAXCUT}\left( G\left( n, \frac{c}{n} \right) \right)$ converges in probability to a single value $MC(c)$ (as $n$ tends to infinity). They also established a similar result for the random regular graph model $G_{n,r}$. Asymptotic bounds on $MC(c)$ were obtained by Coppersmith et al. (see \cite{MAXCUT}), Gamarnik and Li (see \cite{GL}) and Feige and Ofek (see \cite{FeigeOfek}). All these bounds are of the form 
$\frac{c}{4} + \alpha \sqrt{c} + o(\sqrt{c}) 
\leq MC(c) \leq \frac{c}{4} + \beta \sqrt{c} + o(\sqrt{c})$, where $\alpha,\beta$ are known absolute constants and the little-$o$ notation is with respect to $c$. Recently, Dembo, Montanari and Sen found the correct asymptotic behavior of $MC(c)$ up to an error of $o(\sqrt{c})$; they proved that 
$MC(c) = \frac{c}{4} + \gamma \sqrt{\frac{c}{4}} + o(\sqrt{c})$, where $\gamma \approx 0.7632$. They also obtained a similar result for random regular graphs (see \cite{DMS}). 

The above results cover the case when $c$ is large. We, however, focus on the range around the phase transition value $p= \frac{1}{n}$. It is known that the typical structure of $G\left( n, p \right)$ changes significantly as $p$ increases above this value; the giant component appears together with other graph properties. $\text{MAXCUT}$ also has a phase transition at $p = \frac{1}{n}$. For 
$G \sim G\left( n, \frac{c}{n} \right)$, the value of  
$\dist(G) = |E(G)| -  \text{MAXCUT}(G)$ is $O(1)$ in expectation if $c < 1$, and typically $\Omega(n)$ if $c > 1$. Indeed, if $c < 1$ then \whp every connected component in $G \sim G\left( n, \frac{c}{n} \right)$ is either a tree or unicyclic, and the number of unicyclic components has a Poisson limiting distribution with an expected value of $O(1)$ (see Section 5.4 in \cite{Bol98}). This means that in expectation, the distance of $G$ from bipartiteness, which is at most the number of cycles, is $O(1)$. For $c > 1$, a typical $G \sim G\left( n, \frac{c}{n} \right)$ contains a complex giant connected component whose $2$-core is of linear size in $n$ (see Section 5.4 in \cite{JLR}). 

 Coppersmith, Gamarnik, Hajiaghayi and Sorkin \cite{MAXCUT} considered the regime 
$p = \frac{1 + \varepsilon}{n}$ for a fixed $\varepsilon > 0$, and showed that a typical 
$G\sim G(n,p)$ satisfies
$\dist(G) = \Omega\left(\frac {\varepsilon^3}{\log (1/\varepsilon)} \right)n$. They also conjectured that the dependence on $\varepsilon$ can be improved to $\Theta(\varepsilon^3)$. Moreover, they showed that the expectation of $\dist(G)$ is $O(\varepsilon^3)n$. We prove that this conjecture is indeed true.
\begin{theorem}\label{thm:main2}
	Let $\varepsilon \in (0,1)$ and let $G \sim G\left( n, \frac{1+\varepsilon}{n}  \right)$. Then \whp $\dist(G) = \Theta (\varepsilon^3)n$.  
\end{theorem}


The regime $p = \frac{c}{n}$ for $c > 1$, which is considered in this paper, is called the {\em strictly supercritical regime}. We note that the problem of finding the typical distance to bipartiteness has also been considered in the following regimes: the strictly subcritical regime $p = \frac{c}{n}$ for $c < 1$; the subcritical regime
$p = \frac{1 - \mu}{n}$ for $n^{-1/3}\ll\mu\ll 1$; the scaling window 
$p = \frac{1 \pm \mu}{n}$ for 
$\mu = \Theta\left( n^{-1/3} \right)$; the supercritical regime $p=\frac {1+\mu}n$ for $n^{-1/3}\ll\mu\ll 1$. For all these regimes, Daud\'e, Mart\'inez,  Rasendrahasina and Ravelomanana found the limit distribution of a normalization of $\dist(G)$ (see \cite{MaxCutSupercritical} for the details). For the subcritical regime and the scaling window, Scott and Sorkin \cite{ScottSorkin} showed that a maximum cut of a random graph can be found in linear expected time by a simple algorithm.

We note that the MAX $k$-CUT problem was also studied in several models of random graphs, such as $G(n,p)$, $G(n,m)$, and the random regular graph model $G_{n,r}$ (see, \cite{Bertone, MAXCUT,Coja,Kala}).


The key ingredient in the proof of Theorem \ref{thm:main2} is the characterization of the giant component of random graphs in the strictly supercritical regime, obtained by Ding, Lubetzky and Peres in \cite{DLP}. 
 
We use Theorem \ref{thm:main2} to address two supposedly unrelated problems.  
The first problem is related to the chromatic number of random graphs and homomorphisms of random graphs. We first give relevant definitions for general graphs. A graph $G$ is called $k$-colorable if there exists a coloring $c:V(G)\to [k]=\{1,2,\dots,k\}$ such that if $v\sim u$ then $c(v)\neq c(u)$. The \textit{chromatic number} of a graph $G$, denoted by $\chi(G)$, is the minimal $k$ for which $G$ is $k$-colorable.  A \textit{graph homomorphism} $\varphi$ from a graph $G=(V,E)$ to a graph $G'=(V',E')$, is a (not necessarily injective) mapping $\varphi:V\to V'$ from the vertex set of $G$ to the vertex set of $G'$ such that if $\{v,u\}\in E(G)$ then $\{\varphi(v),\varphi(u)\}\in E(G')$. We simply denote this mapping by $\varphi:G \to G'$. By the definition, we can see that a graph $G$ is $k$-colorable if and only if there exists a homomorphism $\varphi: G\to K_k$, where $K_k$ is the complete graph on $k$ vertices. For a graph $G$ and an integer $\ell\geq 1$, a homomorphism $\varphi: G\to C_{2\ell +1}$ implies a homomorphism $\varphi':G\to C_{2k+1}$ for every $k\in [\ell] = \{1,...,\ell\}$, where $C_r$ is the cycle with $r$ vertices. In particular, a graph $G$ is $3$-colorable if there exists a homomorphism $\varphi:G\to C_{2\ell +1}$ for some integer $\ell\geq 1$ (as this implies a homomorphism $\varphi':G\to C_3$). The opposite direction is not always true. For example, the graph $G=C_3$ is 3-colorable but is not homomorphic to any $C_{2\ell+1}$ for $\ell\geq 2$. Thus, by considering homomorphisms from a graph $G$ to $C_{2\ell+1}$, we can measure how ``strong" the property of 3-colorability is in a given graph $G$.   

Here we will consider this notion in the random graph model $G(n,p)$ (see, e.g., Chapter 7 of \cite{JLR} for a detailed overview of coloring properties of random graphs). In the case of random graphs, it is known that for $p= \frac cn$ where $c>1$ we have that \whp $G\sim G(n,p)$ is not 2-colorable, that is, \whp $\chi(G)\geq 3$ (see, e.g., \cite{Bol98,JLR}).
 In their paper \cite{Between23}, Frieze and Pegden proved the following.

\begin{theorem}\label{thm:Between23}
 For any $\ell > 1$, there is an $\varepsilon > 0$ such that with high probability, $G\sim G(n, \frac {1+\varepsilon}n)$
	either has odd-girth $< 2\ell + 1$ or has a homomorphism to $C_{2\ell+1}$ (the cycle of length $2\ell+1$).
\end{theorem}

In Theorem \ref{thm:Between23} the size of the cycle, $2\ell +1$, is fixed, and $\varepsilon$ (and thus the edge probability $p$) depends on $\ell$. It is also natural to ask, for a fixed probability, about the values of $\ell$ for which there is a homomorphism from the random graph to $C_{2\ell + 1}$. 
Frieze and Pegden conjectured the following.

\begin{conjecture}[Conjecture 1 in \cite{Between23}]\label{conj:Between23}
	For any $c > 1$, there is an $\ell_c$ such that with high probability, there is no
	homomorphism from $G\sim G(n, \frac {c}n)$
	to $C_{2\ell+1}$ for any $\ell \geq \ell_c$.
\end{conjecture}

\noindent
In Section \ref{sec:between23} we give a proof of this conjecture using Theorem \ref{thm:main2}. We show the following.

\begin{theorem}\label{thm:homo}
		For any $\varepsilon > 0$, there is an $\ell_\varepsilon$ such that \whp there is no
		homomorphism from $G\sim G(n, \frac {1+\varepsilon}n)$
				to $C_{2\ell+1}$ for any $\ell \geq \ell_\varepsilon$. In fact, $\ell_\varepsilon=O \left(\frac 1{\varepsilon^3}\right)$.
\end{theorem}

The second application of Theorem \ref{thm:main2} is related to colorability of biased random tournaments. Consider the following random tournament model. We start with $K_n$ and order its vertices in the natural order $1,2,3,\dots , n$. A \textit{$p$-random tournament} on $n$ vertices, $T\sim T(n,p)$, is a tournament for which $\overrightarrow{ji}\in E(T)$ with probability $p$ and $\overrightarrow{ij}\in E(T)$ with probability $1-p$, for every $1\leq i<j\leq n$ independently. Observe trivially that for $p=o(1)$ most of the edges of $T\sim T(n,p)$ typically point forward, explaining our terminology of a biased random tournament. In \cite{Tnp},  \L uczak,  Ruci\'nski, and  Gruszka studied properties of $T(n,p)$ such as the appearance of small subdigraphs and  strong connectivity.  In some cases, it is more natural to view this model as a perturbation of the transitive tournament: we start with the transitive tournament on $[n]$ and then choose each oriented edge with probability $2p$ and re-orient it uniformly at random. 
Random perturbations of general tournaments were studied by Krivelevich, Kwan and Sudakov in \cite{PerturbedDigraphs}.

We say that a tournament is $k$-colorable if there a $k$-coloring of its vertices such that for every $i\in [k]$ the sub-tournament induced by the vertices of color $i$ is transitive. The \textit{chromatic number} of a tournament $T$, denoted by $\chi(T)$, is the minimal $k$ for which $T$ is $k$-colorable. In the past few years there has been extensive research into the chromatic number of tournaments. Much of the work dealt with the chromatic number of tournaments with some forbidden substructure. Most notably, Berger et al. characterized the tournaments that are {\em heroes} (see \cite{heroes}). A tournament $H$ is called a hero if there exists $C > 0$ such that every $H$-free tournament $G$ satisfies $\chi(G) \leq C$. 

In this paper we will show that the coloring properties of $p$-random tournaments are similar to those of the random graph model $G(n,p)$. For random graphs, it is known that for 
$k \geq 3$ we have a sharp threshold for $G\sim G(n,p)$ being $k$-colorable. In particular, it is known that for $c>0$, $G\sim G(n,\frac cn)$ satisfies \whp $\chi(G)\in \{k,k+1\}$, where $k$ is the smallest integer such that $c<2k\log k$ (see \cite{AchFri99,AchNaor04,Luczak91}). However, in the case of 2-colorability we observe an entirely different phenomenon (see, e.g., Chapter 5 in \cite{Bol98}). For $c\in(0,1)$, $G\sim G(n,\frac cn)$ contains an odd cycle with probability bounded away from zero (and dependent on $c$), and therefore $\chi(G)>2$ with probability bounded away from zero. On the other hand, $G$ is acyclic with probability bounded away from zero (and dependent on $c$), and thus $\chi(G)\leq 2$ with probability bounded away from zero. 
For $c>1$, $G\sim G(n,\frac cn)$ is not $2$-colorable \whp 
We show that $p$-random tournaments behave similarly.

	\begin{theorem}\label{thm:2colorTour}
		Let $\varepsilon \in (0,1)$ and let $T\sim T(n,\frac {1-\varepsilon}n)$. Then for large enough $n$ we have that \newline  
		$c_{\varepsilon} \leq \Pr\left[\chi(T)\leq2\right] \leq 1-c'_{\varepsilon}$, where 
		$c_{\varepsilon}, c'_{\varepsilon}>0 $ are constants depending on $\varepsilon$.  
	\end{theorem}

For the regime $p=\frac {1+\varepsilon}n$, we prove an analogue  of Theorem \ref{thm:main2} for tournaments. The {\em distance of $T$ from bipartiteness}, denoted by $\distT(T)$, is the minimal number of edges that need to be reversed to make $T$ bipartite (2-colorable). The following theorem is the tournament analogue of Theorem \ref{thm:main2}.

\begin{theorem}\label{thm:TourFar}
	Let $\varepsilon\in (0,1)$ 
	and let $T \sim T \left( n, \frac{1+\varepsilon}{n} \right)$. Then 
	\whp $\distT(T) = \Theta(\varepsilon^3) n$. 
\end{theorem}
\noindent
Theorem \ref{thm:TourFar} clearly implies the following corollary.

\begin{corollary}\label{thm:TourNot2color}
	Let $\varepsilon\in (0,1)$ and let $T\sim T(n,\frac {1+\varepsilon}n)$. Then \whp 
	$\chi(T)>2$.  
\end{corollary}

In the next theorem we determine the order of magnitude of the threshold for $k$-colorability for every $k\geq 3$.

 \begin{theorem}\label{prop:kColorabilityOfRandomTournaments}
 	For every $k\geq 3$, there exist constants $c:=c(k)$ and $C:=C(k)$ such that if $p\geq \frac {C(k)}n$, then for $T\sim T(n,p)$ \whp $\chi(T)>k$, and if $p\leq \frac {c(k)}n$ then for $T\sim T(n,p)$ \whp $\chi(T)\leq k$. In fact, $c(k),C(k)=\Theta(k\log k)$.
 \end{theorem}
%

\begin{remark}\label{re:smallPro}
For $T\sim T(n,p)$, where $p=o(\frac 1n)$, \whp $\chi(T)\leq 2$. This will be argued later.
\end{remark}

\begin{remark}
	Theorems \ref{thm:homo} and \ref{thm:TourFar} could have been proved  (with worse asymptotics in $\varepsilon$) using a weaker bound than the one in Theorem \ref{thm:main2}. In particular,  any linear (in $n$) lower bound on $\dist$, such as the bound given by Coopersmith et al.\ in \cite{MAXCUT}, would imply weaker (quantitative) versions of these theorems. 
\end{remark}		

\subsection{Notation and terminology}\label{sec:notation}

Our graph-theoretic notation is standard and follows that of \cite{West}. In particular we use the following:
For a graph $G$, let $V=V(G)$ and $E=E(G)$ denote its set of
vertices and edges, respectively. We let $v(G)=|V|$ and $e(G)=|E|$. For a subset $U\subseteq V$ we
denote by $E_G(U)$  all the edges
$e\in E$ with both endpoints in $U$. For subsets $U,W\subseteq V$ we
denote by $E_G(U,W)$  all the edges
$e\in E$ with both endpoints in $U\cup W$ for which $e\cap U\neq
\emptyset$ and $e\cap W\neq \emptyset$. We simply write $E(U)$ or $E(U,W)$ in the cases where there is no risk of confusion. We also write $e_G(U)=|E_G(U)|$ and $e_G(U,W)=|E_G(U,W)|$. 


We assume that $n$ is large enough where needed.  We say that an event holds \emph{with high probability} (w.h.p.) if its probability  tends to one as $n$ tends to infinity.  For the sake of simplicity and clarity of presentation, and in order
to shorten some of the proofs, no real effort is made to optimize
the constants appearing in our results. We also sometimes omit floor
and ceiling signs whenever these are not crucial. When we write $\log n$ we mean the natural logarithm.

A tournament $T$ on $[n]$ is an orientation of the complete graph $K_n$. That is, $V(T)=[n]$ and for every edge $\{i,j\}$ of $K_n$ either $(i,j)\in E(T)$ or $(j,i)\in E(T)$.
We usually write $\overrightarrow{ij}$ to mean $(i,j)\in E(T)$ (the edge $\{i,j\}$ appears with the orientation from $i$ to $j$). 
Let $U,W\subseteq V(T)$ be two disjoint subsets of vertices. We write $U\to W$ to mean that for every $u\in U$ and for every $w\in W$, $\arc{uw}\in E(T)$. In the case that $U=\{u\}$ or $W=\{w\}$ we simply write $u\to W$ or $U\to w$, respectively. 
For $U\subseteq [n]$, we let $T[U]$ be the subtournament of $T$ induced by $U$.

For a tournament $T$ with $V(T)=[n]$ we let $B = B(T)$ be the undirected graph obtained from $T$ by keeping only backedges, that is $V(B)=[n]$ and 
$E(B) = \{\{i,j\} \ | \ i<j \text{ and } \arc{ji} \in E(T)\}$. $B$ is called the 
{\em backedge graph} of $T$. 
We will often use the following simple observation.

\begin{observation}\label{obs:backedge_chromatic}
Let $T$ be a tournament, let $B$ be its backedge graph and let 
$E' \subseteq E(B)$. If deleting the edges in $E'$ makes $B$ a $k$-colorable graph then reversing the corresponding arcs to $E'$ (in $T$) makes $T$ a $k$-colorable  tournament. In particular, if $B$ is $k$-colorable (as a graph) then $T$ is $k$-colorable (as a tournament) and $\distT(T) \leq \dist(B)$. 
\end{observation}

\begin{proof}
Set $B' = (V(B),E(B)\setminus E')$. Let $T'$ be the tournament obtained from $T$ by reversing every edge in $E'$. It is easy to see that $B'$ is the backedge graph of $T'$. 
Every independent set in $B'$ spans a transitive tournament in $T'$, implying that $T'$ is k-colorable.  
\end{proof}



\section{Tools}

\subsection{Binomial distribution bounds}
We use the following standard bound on the lower and the
upper tails of the Binomial distribution due to Chernoff (see, e.g.,
\cite{AloSpe2008}, \cite{JLR}):

\begin{lemma}\label{Che}
	Let $X\sim \Bin(n,p)$ and $\mu=\mathbb{E}(X)$, then
	\begin{enumerate}
		\item $\Pr\left(X<(1-a)\mu\right)<\exp\left(-\frac{a^2\mu}{2}\right)$ for every $a>0.$
		\item $\Pr\left(X>(1+a)\mu\right)<\exp\left(-\frac{a^2\mu}{3}\right)$ for every $0 < a < 1.$
	\end{enumerate}
\end{lemma}


\noindent
We will also use the following Chernoff-type bound due to Hoeffding  (see, e.g.,
\cite{AloSpe2008}):

\begin{lemma}\label{Hof}
	For $X\sim \Bin(n,p)$ and $\mu=\mathbb{E}(X)$, we have 
		 $\Pr\left(|X-\mu|>tn\right)<2e^{-2t^2n}.$

\end{lemma}

\subsection{The structure of the giant component in random graphs }

For the proof of Theorem \ref{thm:main2} we will use a theorem by Ding, Lubetzky and Peres. First we need the following definition. 

\begin{definition}
	Let $G=(V,E)$ be a graph. The \textit{2-core} of $G$ is the maximal induced subgraph of $G$ with minimum degree at least two.
\end{definition}

\begin{theorem}[Theorem 1 in \cite{DLP}]\label{DLP}
	
	Let $\mathcal{C}$ be the 2-core of the largest component of $G(n,p)$ for $p = \frac{\lambda}{n}$, where $\lambda = 1+\varepsilon$ and $\varepsilon \in (0,1)$ is fixed. Let $\mu < 1$ be such that $\mu e^{-\mu} = \lambda e^{-\lambda}$. Let $\mathcal{\tilde{C}}$ be the following model:
	\begin{enumerate}
		\item 
		Let $\Lambda $ be Gaussian $ \mathcal N(\lambda - \mu, 1/n)$ and let $D_u \sim Poisson(\Lambda)$ for $u \in [n]$ be i.i.d., conditioned on the event that $\sum_{u=1}^{n}{D_u 1_{D_u \geq 3}}$ is even. Let $N_k = \#\{u : D_u = k\}$ and $N = \sum_{k\geq 3}{N_k}$. Select a random multigraph $\mathcal{K}$ on $N$ vertices, uniformly among all multigraphs (possibly with loops) that have $N_k$ vertices of degree $k$ for every $k \geq 3$. 
		\item Replace the edges of $\mathcal{K}$ by internally disjoint paths of i.i.d $Geom(1 - \mu)$ lengths. 
	\end{enumerate}
	Then $\mathcal{C}$ is contiguous to the model $\mathcal{\tilde{C}}$, that is, if $\Pr[\tilde{\C} \in \mathcal{A}] \rightarrow 0$ then $\Pr[\C \in \mathcal{A}] \rightarrow 0$ for any set of graphs $\mathcal{A}$. 
	
\end{theorem}


By $\text{Geom}(1 - \mu)$ we mean a  random variable which takes the value $k$ with probability $\mu^{k-1}(1 - \mu)$ for every $k \geq 1$. 
Throughout this section, the notation is the same as in Theorem \ref{DLP}. 
We assume that $\varepsilon$ is small enough where needed.  

\begin{claim}\label{claim:mu}
$1 - \varepsilon < \mu < 1 - \varepsilon + \frac{2}{3}\varepsilon^2$. In particular 
$\mu = 1 - \varepsilon + O(\varepsilon^2)$.
\end{claim}

\begin{proof}
	First we show that $\mu>1-\varepsilon$. Let $f:\mathbb R \to \mathbb R$ be the function $f(x)=xe^{-x}$. Then $f'(x)=(1-x)e^{-x}$ and $f'(x)>0$ for $x\in (0,1)$. Therefore, $f$ is increasing in $(0,1)$. We need to show that $f(\mu)>f(1-\varepsilon)$.	Since $f(\mu) =  f(1+\varepsilon)$, it is enough to show that $f(1-\varepsilon) < f(1+\varepsilon)$. This is equivalent to $\frac{1-\varepsilon}{1+\varepsilon}\cdot e^{2\varepsilon}< 1$, which is true for every $\varepsilon>0$. 
	
		Now we show will that $\mu < 1-\varepsilon+\frac 23\varepsilon^2$. 
		Since $\varepsilon \in (0,1)$ we have $1-\varepsilon+\frac 23\varepsilon^2<1$. 
		If we will show that $f(\mu) < f(1-\varepsilon+\frac 23\varepsilon^2)$, the claim will follow. Since $f(1+\varepsilon)=f(\mu)$, we need to show that $(1+\varepsilon)e^{-1-\varepsilon} < (1-\varepsilon+\frac 23\varepsilon^2)e^{-(1-\varepsilon+\frac 23\varepsilon^2)}$, which is equivalent to $\frac {1+\varepsilon}{1-\varepsilon+\frac 23\varepsilon^2} < e^{2\varepsilon-\frac 23\varepsilon^2}$.  This can be verified for every $\varepsilon\in (0,1)$ by taking $\log$ of both sides and differentiating. 
\end{proof}

\begin{corollary}\label{cor:Lambda}
holds \whp
\end{corollary}
\begin{proof}
By Claim \ref{claim:mu} we have 
$\mathbb{E}[\Lambda] = \lambda - \mu = 
1 + \varepsilon - (1 - \varepsilon + O(\varepsilon^2)) = 2\varepsilon + O(\varepsilon^2)$. Since $\text{Var}[\Lambda] = \frac{1}{n}$, Chebyshev's inequality gives that $\Pr\left[ |\Lambda - \mathbb{E}[\Lambda]| \geq \varepsilon^2 \right] \underset{n \rightarrow \infty}{\longrightarrow} 0$. 
\end{proof}

\begin{claim}\label{claim:cond_even_sum}
Suppose that $\Lambda \in (0,1)$. Then
$\Pr\left[ \sum_{u=1}^{n}{D_u 1_{D_u \geq 3}} \text{ is even}\right] \geq \nolinebreak \frac{1}{2}$. 
\end{claim}
\begin{proof}
For $1 \leq m \leq n$, define $p_m$ to be the probability that 
$\sum_{u=1}^{m}{D_u 1_{D_u \geq 3}}$ is even.
We prove by induction on $m$ that $p_m \geq \frac{1}{2}$. Evidently, $p_1$ is the probability that $\text{Poisson}(\Lambda)$ is smaller than $3$ or even. Since $\Lambda \in (0,1)$, we have 
$\Pr[\text{Poisson}(\Lambda) = k] > \Pr[\text{Poisson}(\Lambda) = k+1]$ for every $k$, implying that $p_1 \geq \frac{1}{2}$. For $m \geq 2$, by conditioning on the value of $D_m$ we get 
$$p_m = 
\Pr[ D_m < 3 \text{ or } D_m \text{ even} ] \cdot p_{m-1} + 
\Pr[D_m \geq 3 \text{ and odd}] \cdot (1-p_{m-1}) = 
p_1p_{m-1} + (1-p_1)(1-p_{m-1}) 
.$$
By the induction hypothesis $p_{m-1} \geq \frac{1}{2}$ we get 
$p_m = (2p_1 - 1)p_{m-1} + 1 - p_1 \geq \frac{1}{2}$. 
\end{proof}
\begin{corollary}\label{cor:size_of_K}
\Whp we have 
$e(\K) = (2\varepsilon^3 + O(\varepsilon^4) )n$.
\end{corollary}
\begin{proof}
We condition on the value of $\Lambda$ and assume that 
$\Lambda = 2\varepsilon + O(\varepsilon^2)$, which holds \whp by Corollary \ref{cor:Lambda}. For $D \sim \text{Poisson}(\Lambda)$ we have 
$$\mathbb{E}[D \cdot 1_{D_\geq 3}] = e^{-\Lambda} \sum\limits_{k=3}^{\infty}{\frac{\Lambda^k}{k!}k} = (1+O(\Lambda)) \left( 0.5\Lambda^3 + O(\Lambda^4)\right) = 0.5\Lambda^3 + O(\Lambda^4) = 4\varepsilon^3 + O(\varepsilon^4).$$   

Set $X := \sum_{u=1}^{n}{D_u 1_{D_u \geq 3}}$.
We have 
$\mathbb{E}[X] = (4\varepsilon^3 + O(\varepsilon^4) )n$. Note that $\text{Var}(D_u1_{D_u\geq 3})=O(1)$ and by independence $\text{Var}(X)=O(n)$. By  Chebyshev's inequality we have 
$$\Pr[|X - \mathbb{E}[X]| > \varepsilon^4 n] \leq \frac {\text{Var}(X)}{\varepsilon^8n^2} = o(1),$$
 and this bound is uniform for all $\Lambda$ in the range $2\varepsilon + O(\varepsilon^2)$.
Since $\Pr\left[ X \text{ is even}\right] \geq \frac{1}{2}$ (by Claim \ref{claim:cond_even_sum}) we have 
$$\Pr\big[ |X - \mathbb{E}[X]| > \varepsilon^4 n \; \big| \; X \text{ is even}\big] \leq \frac {\Pr\big[ |X - \mathbb{E}[X]| > \varepsilon^4 n \big]} {\Pr [X \text{ is even}]} =  o(1).$$
Thus $\Pr\big[ X = ( 4\varepsilon^3 + O(\varepsilon^4) )n \; \big| \; X \text{ is even}\big] = 1 - o(1)$. Since $e(\K) = \frac{X}{2}$ conditioned on $X$ being even (recall Theorem \ref{DLP}), we are done.
\end{proof}

\section{Proof of Theorem \ref{thm:main2}}

\noindent 
The following lemma is the main part of the proof of the lower bound of Theorem \ref{thm:main2}.

\begin{lemma}\label{lem:main}
	Let  $\K$ be a multigraph satisfying $e(\K) \geq \frac {3}{2}v(\K)$ and let $0.99\leq \mu<1$. 
	Replace the edges of $\mathcal{K}$ by paths of i.i.d. $Geom(1 - \mu)$ lengths and denote this new multigraph by $\tilde\C$. Then \whp we have 
	$\dist ( \tilde{\C} ) \geq 0.001 e(\K)$.
\end{lemma}

\begin{proof}
	Denote by $\ell_e$ the length of the path that replaces the edge 
	$e \in E(\K)$ in $\tilde{\C}$, and by $P_e$ the set of edges of this path. Note that  
	$$p_{even}:=\Pr(\ell_e\ is\ even)=
	\sum_{k = 1}^{\infty}\mu^{2k-1}(1-\mu)= \frac {\mu}{1+\mu} < \frac{1}{2}$$ and
 $$p_{odd}:=\Pr(\ell_e\ is\ odd)=1-\Pr(\ell_e\ is\ even)=\frac {1}{1+\mu}<0.51.$$
  The last inequality holds by our assumption that $\mu \geq 0.99$.
We conclude that $p_{odd},p_{even} \in (0.49,0.51)$.
	
	It is clear that if $\dist( \tilde{\C} ) < 0.001 e(K)$ then there is a bipartition 
	$V(\tilde{\C}) = \tilde{V_1} \uplus \tilde{V_2}$ satisfying 
	$e(\tilde V_1)+e(\tilde V_2) < \nolinebreak 0.001 e(K)$.
	Let $V(\K) = V_1\uplus V_2$ be a bipartition of $\K$. We will bound the probability that $V(\tilde{\C})$ has a bipartition $\tilde{V_1} \uplus \tilde{V_2}$ that extends $V_1\uplus V_2$ (i.e. $V_1 \subseteq \tilde{V_1}$ and $V_2 \subseteq \tilde{V_2}$) and satisfies 
	$e(\tilde V_1)+e(\tilde V_2)< 0.001 e(K)$. 
	
	\noindent
	Let $e\in \K$. Consider the following events:
	\begin{enumerate}[label=(\alph*)]
	\item $e \in E(V_1) \cup E(V_2)$ (that is, $e$ lies inside $V_1$ or inside $V_2$) and $\ell_e$ is odd.
	\item $e \in E(V_1,V_2)$ and $\ell_e$ is even.
	\end{enumerate}
	Observe that if (a) or (b) happen then for every bipartition $V(\tilde{\C}) = \tilde{V_1} \uplus \tilde{V_2}$ which extends 
	$V_1 \uplus V_2$, at least one of the edges in $P_e$ must lie inside either $\tilde{V_1}$ or $\tilde{V_2}$. 
	If (a) or (b) happen then we call $e$ \textit{bad}.
	We conclude that for any bipartition 
	$V(\tilde{\C}) = \tilde V_1 \uplus \tilde V_2$ which extends $V_1 \uplus V_2$, $e(\tilde V_1)+e(\tilde V_2)$ is at least the number of bad edges of $\K$. Let us denote this number by $X$. Since $p_{odd},p_{even} \geq 0.49$, every $e \in E(\K)$ is bad with probability at least $0.49$. Thus, $X$ stochastically dominates the distribution $\Bin(e(\K), 0.49)$. By Lemma \ref{Hof} (with $t=0.489$) we have
	$$ \Pr\left[X < 0.001 e(K)\right] \leq 
	\Pr\left[\Bin(e(\K), 0.49) < 0.001 e(K)\right] \leq 
	 2e^{-2 \cdot (0.489)^2e(\K)}.$$
By the union bound over all partitions $V_1\uplus V_2$ of $\K$ we get: the probability that there exists a partition $\tilde V_1\uplus \tilde V_2$ of $\tilde{\C}$ with $e(\tilde V_1)+e(\tilde V_2) < 0.001 e(\K)$ is at most
$$2^{v(\K)} \cdot 2e^{-2(0.489)^2e(\K)} \leq 
2e^{\left( \frac{2\log 2}{3} - 2 \cdot (0.489)^2 \right) e(\K)}=o(1).$$ 
In the first inequality we used the assumption that $e(\K) \geq \frac{3}{2}v(\K)$.
%
%
%
%
%
\end{proof}

\begin{proof}[Proof of Theorem \ref{thm:main2}]

 We can assume that $\varepsilon < \varepsilon_0$ for some small constant $\varepsilon_0$. 
 Indeed, if we know the theorem for some $\varepsilon_0$ then for every 
 $\varepsilon \in (\varepsilon_0,1)$ and $G \sim G\left( n, \frac{1+\varepsilon}{n} \right)$ we have (by monotonicity) \whp 
 $\Omega(\varepsilon_0^3)n \leq \dist(G) \leq e(G) \approx \frac{1+\varepsilon}{2}n < n$. So by adjusting the implied constants in the $\Theta$-notation, we get the theorem for $\varepsilon > \varepsilon_0$ as well.  

 Consider the model $\tilde{\C}$ from Theorem \ref{DLP}. We prove that \whp
 $\dist( \tilde{\C} ) = \Theta(\varepsilon^3)n$. 
 Let $\K$ be the multigraph generated in item 1 of Theorem \ref{DLP}. By Claim \ref{cor:size_of_K} we have \whp $e(\K) = \Theta(\varepsilon^3) n$. Since all vertex-degrees in $\K$ are at least 3, we have $e(\K) \geq \frac{3}{2}v(\K)$. By Claim \ref{claim:mu}, $\mu \geq 0.99$ if $\varepsilon$ is sufficiently small. By Lemma \ref{lem:main} we get that \whp
 $\dist( \tilde{\C} ) \geq 0.001 e(\K) = \Omega(\varepsilon^3)n$. 
 
We now show that \whp $\dist( \tilde{\C} ) = O(\varepsilon^3)n$. For each $e \in E(\K)$, let $P_e$ be the path in $\tilde{\C}$ which replaces $e$ and let $e^*$ be an arbitrary edge of $P_e$. It is easy to see that for every cycle $C$ in $\tilde{\C}$ there is $e \in E(\K)$ such that $C$ contains all the edges of $P_e$. Moreover, if $C$ is an odd cycle then there is such $e \in E(\K)$ for which $|P_e|$ is odd. Thus, removing 
$\{e^* : e \in E(\K)\}$ from $\tilde{\C}$ leaves an acyclic graph, and removing 
$\{e^* : e \in E(\K), \; |P(e)| \text{ odd}\}$ leaves a bipartite graph. As in the proof of Lemma \ref{lem:main}, we have 
$p_{odd} := \Pr[ |P_e| \text{ is odd} ] = \frac{\mu}{1 + \mu} < \frac{1-\varepsilon + \frac{2}{3}\varepsilon^2}{2-\varepsilon} \leq \frac{1}{2} - \frac{\varepsilon}{8}$. Here we used Claim \ref{claim:mu} and assumed that $\varepsilon$ is small enough. Therefore, the random variable 
$X = \left| \left\{ e \in E(\K) : |P_e| \text{ is odd} \right\} \right|$ is stochastically dominated by $\Bin(e(\K),\frac{1}{2} - \frac{\varepsilon}{8})$. By Chernoff's Inequality (Lemma \ref{Che}), we have 
$\Pr[X > \frac{1}{2}e(\K)] = o(1)$. Thus, \whp
$\dist( \tilde{\C} ) \leq X \leq \frac{1}{2}e(\K) = O(\varepsilon^3)n$, using Corollary \ref{cor:size_of_K}.    

Now let $G \sim G\left( n, \frac{1+\varepsilon}{n} \right)$ and let $\C$ be the 2-core of the largest component of $G$. 
By Theorem \ref{DLP} and the fact that \whp $\dist( \tilde{\C} ) = \Theta(\varepsilon^3)n$ we get that \whp $\dist( \C ) = \Theta(\varepsilon^3)n$.
This already establishes that \whp
$\dist (G) = \Omega(\varepsilon^3)n$. 
To prove the likely upper bound on $\dist (G)$, we use the known fact that the expected number of cycles in $G$ not contained in the largest component is $O(1)$ (see, e.g., \cite{JLR}). Thus, \whp the number of these cycles is $o(n)$, implying that all components other than the largest can be made acyclic by omitting $o(n)$ edges. Moreover, \whp the 2-core of the largest component, $\C$, can be made bipartite by omitting $O(\varepsilon^3)n$ edges. Observe that every edge of the largest component which is not in its 2-core is not contained in any cycle and thus can be added to every cut of the 2-core. Therefore, there is a cut of the largest component, consisting of a maximum cut of the 2-core and all edges that are not in the 2-core. We conclude that $G$ can be made bipartite by omitting $O(\varepsilon^3)n$ edges, as required. 
\end{proof}

%

%

\begin{remark}\label{re:algo}
	\emph{
The proof of the upper bound in Theorem \ref{thm:main2} is algorithmic in the sense that there is
a deterministic polynomial-time algorithm that \whp finds in 
$G\sim G(n,\frac {1+\varepsilon}n)$ a cut of size 
$e(G) - O(\varepsilon^3)n$. 
The algorithm follows the proof of the upper bound: first it finds the connected components of $G$. Then, for every connected component $C$ but the largest, the algorithm greedily deletes one edge to make $C$ bipartite (this is possible because \whp all components but the largest are unicyclic or trees). The algorithm then finds the 2-core of the largest component and deletes one edge from every path (here a cycle is thought of as a path with identical endpoints) in the 2-core whose internal vertices are all of degree $2$. 
The remaining edges form a cut whose size is \whp $e(G) - O(\varepsilon^3)n$. 
}
\end{remark}

\begin{remark}
	\emph{
Using a similar technique, we can obtain an analogue of Theorem \ref{thm:main2} for random graphs in the supercritical phase, i.e. $p = \frac{1 + \varepsilon}{n}$ for $n^{-1/3} \ll \varepsilon \ll 1$. Instead of using Theorem \ref{DLP}, we need to use an analogous result that describes the structure of the giant component of random graphs in the supercritical phase, see \cite{DKLP}. In this manner we can prove that for $\varepsilon$ as above, a typical 
$G \sim  G\left( n, \frac{1 + \varepsilon}{n} \right)$ satisfies 
$\dist \left( G \right) = \Theta(\varepsilon^3)$.
This has already been shown (using a different proof technique) in \cite{MaxCutSupercritical}. 
}
\end{remark}

\section{Proof of Conjecture \ref{conj:Between23}}\label{sec:between23}


\begin{proof}[Proof of Theorem \ref{thm:homo}] 
		Observe that if $G$ is homomorphic to $C_{2\ell+1}$ then 
		$\dist(G) \leq \frac {e(G)}{2\ell+1}$. Indeed, assume that there exists a homomorphism $\varphi:G\to C_{2\ell+1}$. Let $v_1,\dots v_{2\ell +1}$ be the vertices of $C_{2\ell+1}$ and let $e_1,\dots ,e_{2\ell+1}$ be the edges of $C_{2\ell+1}$, where $e_i=\{v_i,v_{i+1}\}$ for every $i\in[2\ell]$ and $e_{2\ell+1}=\{v_{2\ell+1},v_1\}$. Since every edge of $G$ is mapped into one of the $e_i$-s, there exists 
		$i_0 \in \{1,\dots,2\ell + 1\}$ such that $\varphi$ maps at most $\frac{e(G)}{2\ell+1}$ edges to $e_{i_0}$. By erasing the edges mapped to $e_{i_0}$ we turn $G$ into a graph which is homomorphic to a path (of length $2\ell$) and hence bipartite. Thus, $\dist(G) \leq \frac {e(G)}{2\ell+1}$.
		
		Now let $G \sim G\left( n, \frac{1+\varepsilon}{n} \right)$. By Theorem \ref{thm:main2} and since \whp 
		$e(G) = \left( \frac{1+\varepsilon}{2} + o(1) \right)n$, we have that \whp $\dist(G) \geq \delta \cdot e(G)$ for 
		$\delta = \Theta(\varepsilon^3)$. Set 
		$\ell_{\varepsilon} = \frac 1{2\delta}$. Then 
		$\frac{1}{2\ell + 1} < \delta$ for every $\ell \geq \ell_{\varepsilon}$. By the previous paragraph, \whp $G$ is not homomorphic to $C_{2\ell+1}$ for any $\ell \geq \ell_{\varepsilon}$. 		
		
\end{proof}

\section{Proof of Theorem \ref{thm:2colorTour}}

Let $T \sim T\left( n, \frac{1-\varepsilon}{n} \right)$. 
For the lower bound, let $B$ be the backedge graph of $T$ (as defined in section \ref{sec:notation}). Then 
$B \sim G\left( n, \frac{1-\varepsilon}{n} \right)$. It is well known that
$\Pr\left[ \chi\left(G\left( n, \frac{1-\varepsilon}{n} \right)\right) \leq 2\right] \geq c_{\varepsilon} > 0$ (see, e.g., \cite{AchFri99,Bol98}). By Observation \ref{obs:backedge_chromatic} it follows that 
$\Pr[\chi(T) \leq 2] \geq \Pr[\chi(B) \leq 2] \geq c_{\varepsilon}$. 


Using the same argument we can also explain Remark \ref{re:smallPro}. 
Indeed, it is easy to show that $B \sim G(n,p)$ is $2$-colorable \whp for 
$p= o\left(\frac 1n\right)$.

We now prove the upper bound in the theorem. Our strategy is to show that with probability at least $c'_{\varepsilon}$ (for $c'_{\varepsilon}$ to be determined later)
$T \sim T\left( n, \frac{1-\varepsilon}{n} \right)$ contains a small non-$2$-colorable subtournament. We consider the tournament $H$ with vertices ${1,...,7}$ and edges:\\ 
$1\rightarrow 2 \rightarrow 3 \rightarrow 1$, $4 \rightarrow 5 \rightarrow 6 \rightarrow 4$, $\{1,2,3\} \rightarrow \{4,5,6\} \rightarrow 7 \rightarrow \{1,2,3\}$. It is easy to verify the following.
\begin{observation}\label{obs:H}
$H$ has the following properties. 
\begin{enumerate}
\item $H$ is not $2$-colorable. 
\item In the ordering $1,\dots,7$, every $S \subseteq \{1,\dots,7\}$ spans at most $|S|$ backedges. 
\end{enumerate}
\end{observation}

It is worth noting that \cite{Tnp} found, for every fixed oriented graph, the threshold for the appearance of this graph in $T(n,p)$. In our case, the threshold for the appearance of $H$ is $\frac{1}{n}$. We show that for $p = \frac{1-\varepsilon}{n}$ (namely, in the threshold), the probability of containing $H$ is bounded away from $0$. 

For a $7$-tuple $1 \leq u_1<...<u_7 \leq n$, we say that $(u_1,...,u_7)$ is an 
{\em ordered copy} of $H$ in $T$ if the map $i \rightarrow u_i$ is an embedding of $H$ into $T$. Let $X$ be the number of ordered copies of $H$ in $T$. We will show that
$\Pr[X > 0] \geq c'_{\varepsilon} > 0$. 

Fix a $7$-tuple $1 \leq u_1<...<u_7 \leq n$. Since $H$ has $5$ backedges in the ordering $1,...,7$, the probability that $(u_1,\dots,u_7)$ is an ordered copy of $H$ is 
$p^5 (1-p)^{\binom{7}{2}-5}  = \big(1 - o(1) \big)p^5$. 
Therefore
\begin{equation*} 
\mathbb{E}[X] = \binom{n}{7} (1 - o(1))p^5 = 
\frac{\big(1 - o(1) \big) \big(1 - \varepsilon \big)^5n^2}{7!} \geq 
\frac{\big(1 - \varepsilon \big)^5n^2}{10^4}.
\end{equation*} 
We now estimate $\mathbb{E}[X^2]$, the expected number of pairs of (not necessarily distinct) ordered copies of $H$. Fix two $7$-tuples $U_1$ and $U_2$ and put $k = |U_1 \cap U_2|$. 
Observe that for every 7-tuple $U$ there is a unique orientation for which $U$ spans an ordered copy of $H$. Therefore,  there is at most one way to orient $E(U_1) \cup E(U_2)$ so that both $U_1$ and $U_2$ are ordered copies of $H$. 
Let $\ell$ be the number of backedges inside $U_1 \cap U_2$ in this orientation (if it exists). 
By Item 2 in Observation \ref{obs:H} we have $\ell \leq k$. The probability that $U_1$ and $U_2$ are both ordered copies of $H$ is 
$ p^5 (1-p)^{16} p^{5 - \ell} (1-p)^{16 - \left( \binom{k}{2} - \ell \right)} \leq p^{10 - \ell} \leq p^{10 - k}$. 
Therefore 
\begin{equation*} 
\mathbb{E}[X^2] \leq 
\sum_{k = 0}^{7}{\binom{n}{7}\binom{7}{k}\binom{n-7}{7-k}p^{10-k}} \leq 
\sum_{k = 0}^{7}{n^{14-k}p^{10-k}} \leq 
\sum_{k = 0}^{7}{n^4} = 8n^4.
\end{equation*}
By the Paley-Zygmund inequality we get 
\begin{equation*}
\Pr[X > 0] \geq \frac{\mathbb{E}[X]^2}{\mathbb{E}[X^2]} \geq \frac{(1 - \varepsilon )^{10} }{10^9}.
\end{equation*}
We may set, say, 
$c'_{\varepsilon} = \frac{(1 - \varepsilon )^{10} }{10^9}$, and then the assertion of the theorem holds.

\section{Proof of Theorem \ref{thm:TourFar}}

\noindent

	For proving Theorem \ref{thm:TourFar}, we will show that in a typical $T\sim T(n,p)$ one needs to reverse  $\Theta(\varepsilon^3)n$ arcs to make it 2-colorable. We assume by contradiction that there exist $c \varepsilon^3 n$ such edges (where $c$ is a small constant) and by reversing them we get two transitive color classes.  
	Our strategy is first to show, using Theorem \ref{thm:main2}, that for every 2-coloring there are many ``long" monochromatic backedges. Among them we find a large matching of long backedges all contained in the same color class. 
	Then, we claim that since this color class spans a transitive tournament, $T$ must in fact contain many more backedges, which stands in contradiction with the fact that \whp $T$ contains $\Theta(n)$ backedges.

We start with three claims that will be used in the proof of Theorem \ref{thm:TourFar}. Let $T$ be a tournament on $[n]$. 
For $\alpha > 0$, we say that a backedge 
$e = \arc{vu} \in E(T)$ is $\alpha$-\textit{long} if 
$v - u \geq \alpha n$ and $\alpha$-\textit{short} otherwise. 

\begin{claim}\label{claim:shortedges}
Let $T \sim T\left( n,\frac {1+\varepsilon}n \right)$ and let 
$\alpha = \alpha(n) = \omega(n^{-1/2})$. Then the following holds \whp
\begin{enumerate}
	\item The number of backedges in $T$ is
	$\left( \frac{1+\varepsilon}{2} + o(1) \right)n$.
	\item The number of $\alpha$-short backedges in $T$ is at most 
	$2\alpha(1+\varepsilon)n$.  
	\item Every vertex participates in at most $\log n$ backedges.
\end{enumerate}
\end{claim}
\begin{proof}
The number of backedges in $T$ is distributed as
$\Bin\left( \binom{n}{2},\frac{1+\varepsilon}{n} \right)$. Item $1$ follow from Chernoff's inequality (Lemma \ref{Che}) with, say, 
$a = \nolinebreak n^{-1/3}$.
Next, the total number of pairs $\{i,j\}$ such that $i < j < i + \alpha n$ is at most $\alpha n^2$. Thus, the number of $\alpha$-short backedges is stochastically dominated by a random variable with distribution $\Bin \left(\alpha n^2,\frac{1+\varepsilon}{n}\right)$. 
Items $2$ now follows from Chernoff's inequality (Lemma \ref{Che})
with parameter $a = \nolinebreak \alpha(1+\varepsilon)n$. Here we use the assumption that $\alpha(n) = \omega(n^{-1/2})$.  

We now prove Item $3$. For $i \in [n] = V(T)$, the number of backedges containing $i$ is distributed 
$\Bin\left( n - 1, \frac{1+\varepsilon}{n} \right)$. By the union bound, the probability that Item $3$ does not hold is at most
$$
n \binom{n-1}{ \log n} 
\left( \frac{1+\varepsilon}{n} \right)^{ \log n} \leq 
n \left( \frac{e(1+\varepsilon)}{ \log n} \right)^{ \log n} 
\leq n \cdot e^{-2\log n} = o(1). 
$$
In the rightmost inequality above we assumed that $n$ is large enough.
\end{proof}



In the following two claims, by a matching we mean a set of vertex-disjoint edges (in a graph or a tournament). 
\begin{claim}\label{claim:backedges}
Let $T$ be a tournament with $V(T) = [n]$, let $\alpha \in (0,1)$ and let 
$t > 0$ be an integer. Suppose that 
$e_1 =\arc{v_1u_1},\dots,e_{t}=\arc{v_{t}u_{t}} \in E(T)$ are $\alpha$-long backedges, such that $\{e_1,\dots,e_t\}$ is a matching and  the subtournament induced by $\{u_1,v_1,\dots,u_t,v_t\}$ is transitive. Then $T$ contains at least $\binom{\alpha t + 1}{2}$ backedges. 
\end{claim}
\begin{proof}
Sample $k \in [n]$ uniformly at random and let 
$F = \{e_i : u_i \leq k < v_i\}$. For every $1 \leq i \leq t$, since $e_i$ is $\alpha$-long we get $\Pr[e_i \in F] \geq \alpha$. Thus, $\mathbb{E}[|F|] \geq \alpha t$, implying that there is some 
$k \in [n]$ for which $|F| \geq \alpha t$. W.l.o.g.\ assume that $F = \{e_1,\dots,e_r\}$. Note that $v_j > k \geq u_i$ for every 
$1 \leq i,j \leq r$. Since by assumption $\{u_1,v_1,\dots,u_r,v_r\}$ spans a transitive tournament, there
is an ordering $u_{\sigma(1)},\dots,u_{\sigma(r)}$ of $u_1,\dots,u_r$ such that $\arc{u_{\sigma(i)} u_{\sigma(j)}} \in E(T)$ for every 
$1 \leq i < j \leq r$. Since $\arc{v_{\sigma(i)} u_{\sigma(i)}} \in E(T)$ for every $1 \leq i \leq r$, we get from transitivity that 
$\arc{v_{\sigma(i)} u_{\sigma(j)}} \in E(T)$ for every 
$1 \leq i \leq j \leq r$. Since $v_j > u_i$, we get that $\arc{v_{\sigma(i)}u_{\sigma(j)}}$ is a backedge for every 
$1 \leq i \leq j \leq r$. Since $\{e_1,\dots,e_t\}$ is a matching, all these backedges are distinct, giving a total of 
$\binom{r+1}{2} \geq \binom{\alpha t + 1}{2}$ backedges, as required. 
\end{proof}

\begin{claim}\label{claim:matching}
Let $F$ be a set of edges (of a graph or a tournament) such that every vertex participates in at most $d$ of the edges in $F$. Then $F$ contains a matching of size at least $\frac{|F|}{d+1}$.
\end{claim}
\begin{proof}
By Vizing's theorem \cite{vizing64}, there is a proper edge-coloring of $F$ using $d+1$ colors. Since each color class is a matching, there must be a matching of size at least $\frac {|F|}{d+1}$.
\end{proof}


\begin{proof}[Proof of Theorem \ref{thm:TourFar}]
Let $\varepsilon>0$ and let $T\sim T(n,\frac {1+\varepsilon}n)$.   
Let $B$ be the backedge graph of $T$ (as defined in Section \ref{sec:notation}). Then 
$B \sim G\left(n,\frac {1+\varepsilon}n \right)$. The upper bound in the theorem, 
$\distT(T) \leq \nolinebreak O(\varepsilon^3) n$, follows immediately from Theorem \ref{thm:main2} and Observation \ref{obs:backedge_chromatic}. It remains to prove that \whp $\distT(T) \geq \Omega(\varepsilon^3)n$. By Theorem \ref{thm:main2}, \whp for every partition $[n] = V_1 \cup V_2$ there are at least $\delta n$ edges of $B$ in $E(V_1) \cup E(V_2)$, where 
$\delta = \Omega(\varepsilon^3)$. Assume that this event occurs and furthermore that the assertion of Claim \ref{claim:shortedges} holds, applied with parameter (say) $\alpha = n^{-1/6}$. These events occur \whp 
so from now on we argue deterministically and show that for every partition $[n] = V_1 \cup V_2$ of $T$ one must reverse at least 
$(\delta - o(1)) n = \Omega(\varepsilon^3)n$ edges to make both parts transitive. 

Fix a partition $[n] = V_1 \cup V_2$. By our assumption, 
$E(V_1) \cup E(V_2)$ contains at least $\delta n$ edges of $B$. By the definition of the backedge graph, each of these edges corresponds to a backedge of $T$. By Item 2 of Claim \ref{claim:shortedges}, the overall number of $\alpha$-short backedges is $o(n)$, implying that there is a set 
$E_0$ of size $|E_0| = (\delta - o(1))n$ of $\alpha$-long backedges, each contained in $V_1$ or $V_2$. Suppose by contradiction that there is a tournament $T'$, obtained from $T$ by reversing less than 
$|E_0| - 4n^{5/6}$ edges of $T$, in which $V_1$ and $V_2$ are transitive. 
Then there is a set $F \subseteq E_0$ of size $|F| \geq 4n^{5/6}$ such that none of the edges in $F$ was reversed. By Item 3 of Claim \ref{claim:shortedges}, every vertex in $T$ participates in at most $\log n$ backedges (of $T$). Since every edge in $F$ is a backedge of $T$, Claim \ref{claim:matching} gives a matching $M \subseteq F$ of size 
$|M| \geq \frac{|F|}{\log n + 1} \geq \frac{|F|}{2\log n} \geq \frac{2n^{5/6}}{\log n}$. W.l.o.g.\ at least half of the edges of $M$ are inside $V_1$. Denote these edges by 
$e_1 =\arc{v_1u_1},\dots,e_{t}=\arc{v_{t}u_{t}}$, where 
$t \geq \frac{n^{5/6}}{\log n}$. 
Then $e_1,\dots,e_t$ is a matching of $\alpha$-long backedges in $T'$, and the subtournament of $T'$ induced by $\{u_1,v_1,\dots,u_t,v_t\}$ is transitive, as it is contained in $V_1$. By Claim \ref{claim:backedges}, $T'$ contains at least 
$\binom{\alpha t}{2} \geq \binom{n^{2/3}/\log n}{2} = \omega(n)$ backedges. Since $T'$ and $T$ differ on $O(n)$ edges, $T$ must contain $\omega(n)$ backedges, in contradiction to Item $1$ of Claim \ref{claim:shortedges}. We conclude that one must reverse at least $|E_0| - 4n^{5/6} = (\delta - o(1))n$ edges to make $V_1$ and $V_2$ transitive, as required. 
\end{proof}


\section{Proof of Theorem \ref{prop:kColorabilityOfRandomTournaments}}

\noindent
We will use the following fact in the proof of Theorem \ref{prop:kColorabilityOfRandomTournaments}.

\begin{lemma}\label{lem:STS}[Theorem 2 from \cite{STS}]
	The largest collection of edge-disjoint triangles in $K_n$ is of size at least $\frac {n^2}6-\frac n3$.
\end{lemma}

Let $k\geq 3$ and let $C=C(k)$ to be chosen later. Let $T\sim T(n,\frac {C}n)$. If $T$ is $k$-colorable, then there exists a transitive sub-tournament of size at least $\frac nk $. We will show that \whp this is not the case. 
Let $T'$ be a fixed sub-tournament on $n_0=\frac nk$ vertices (keeping the order of the vertices of $T$). By Lemma \ref{lem:STS}, there is a set of size $\frac {n_0^2}7$ of triples of vertices $S=\{\{x_i,y_i,z_i\}\}_{i=1}^{{n_0^2}/7}$ in $T'$ such that for every $i\neq j$ we have $|\{x_i,y_i,z_i\}\cap \{x_j,y_j,z_j\}|\leq 1$.
If $T'$ is transitive, then every triangle from $S$ has to be transitive.  Moreover, the probability that a triangle on a vertex set $\{x,y,z\}\in S$ is transitive is $1-(1-p)^2p-(1-p)p^2\leq 1-\frac {0.9C}n$. Therefore, 
$$\Pr(T'\text{ is transitive})\leq \Pr(\text{every triangle in $S$ is transitive})\leq \left(1-\frac {0.9C}n\right)^{|S|}.$$

Now, using the union bound, the probability that there exists a transitive sub-tournament of size $n_0$ is at most
\begin{align*}
\binom{n}{n_0}\left(1-\frac {0.9C}n\right)^{|S|} &\leq \left(k\cdot e^{-\frac {0.9C}n|S|/n_0}\right)^{n/k}\\
&\leq \left(k\cdot e^{-(\frac {0.9C}n)n_0/7}\right)^{n/k}\\
&= \left(k\cdot e^{-\frac 17\frac {0.9C}k}\right)^{n/k}.
\end{align*}

Thus, if we set $C=8k\log k$ (and thus $p=8k\log k/n$) we have that the probability that there exists a transitive sub-tournament of size $n_0$ is $o(1)$ and therefore \whp $T$ is not $k$-colorable. 

For the lower bound, set $c(k)<2(k-1)\log(k-1)$, let $p \leq \frac {c(k)}n$ and consider 
$T \sim T(n,p)$. Let $B$ be the backedge of $T$. It is known (see, e.g., \cite{AchNaor04}) that for $c(k)$ and $p$ as above, $B \sim G(n,p)$ is $k$-colorable \whp By Observation \ref{obs:backedge_chromatic} we get $\Pr[\chi(T) \leq k] \geq \Pr[\chi(B) \leq k] = 1 - o(1)$. 

\section{Concluding Remarks and Open Problems}

In this paper, we used a deep theorem by Ding, Lubetzky and Peres \cite{DLP} to obtain a sharp upper bound for the maximum cut in sparse random graphs. More precisely, we proved that the maximum cut a typical $G \sim G\left( n, \frac{1+\varepsilon}{n} \right)$ is of size $\Theta(\varepsilon^3)n$. We actually proved that \whp the MAXCUT is at most $(\varepsilon^3 + O(\varepsilon^4)) n$, and with additional arguments, this upper bound can be improved to $(\frac 13\varepsilon^3 + O(\varepsilon^4)) n$. We wonder if $\frac{1}{3}$ is indeed the correct constant.   

Using the aforementioned result on MAXCUT, we immediately obtained a solution to a conjecture of Frieze and Pegden. The conjecture stated that \whp the random graph 
$G \sim G\left( n, \frac{1+\varepsilon}{n} \right)$ does not have a homomorphism into a sufficiently large odd cycle. We actually showed for $\ell = \Theta\left( \varepsilon^{-3} \right)$, there is no homomorphism from a typical $G\left( n, \frac{1+\varepsilon}{n} \right)$ to $C_{2\ell + 1}$. It would be interesting to find the minimal $\ell(\varepsilon)$ with this property. 

Regarding random tournaments, we investigated the typical chromatic number of $T \sim T(n,p)$ (the biased $p$-random tournament) for $p = \Theta\left( \frac{1}{n} \right)$. 
We showed that if $k \geq 3$, then there is a threshold for $k$-colorability of order $\frac{k\log k}{n}$. Moreover, there is a coarse threshold for $2$-colorability: if $p = \frac{c}{n}$ for 
$c < 1$, then the probability of being 2-colorable is bounded away from $0$ and $1$; if $p = \frac{c}{n}$ for $c > 1$ then a typical $T(n,p)$ is not $2$-colorable. In fact, we proved something stronger: for $p = \frac{1+\varepsilon}{n}$, a typical $T \sim T\left(n, p\right)$ satisfies 
$\distT(T) = \Theta(\varepsilon^3)n$. 
This shows that the threshold for $2$-colorability of random tournaments is the same as that of random graphs. 
It would be interesting to determine the threshold for $k$-colorability of random tournaments for $k \geq 3$. Using the same method as in the proof of Theorem \ref{thm:TourFar}, one can show that if \whp the distance of $G \sim G(n,\frac{c}{n})$ from $k$-colorability is at least $C(k)n^{3/4}\log^{1/2}n$ (for some appropriate $C(k)$) then \whp $T \sim T(n,p)$ is not $k$-colorable. This leads us to conjecture that the threshold for $k$-colorability of tournaments is the same as that of graphs. 


\subsection*{Acknowledgements}
\noindent
The authors would like to thank the anonymous referees for their careful reading and  helpful remarks.

\end{document}